\documentclass[a4paper,reqno]{amsart}

\usepackage{amsfonts}
\usepackage{amsmath}
\usepackage[colorlinks]{hyperref}
\usepackage{amssymb}
\usepackage{tikz}
\usepackage{color}
\usepackage[all]{xy}
\usepackage{soul}
\usepackage{enumerate}
\usepackage{bm}
\usepackage{bbm}
\usepackage[normalem]{ulem}
\usepackage{mathrsfs}

\allowdisplaybreaks
\numberwithin{equation}{section}
\newtheorem{theorem}{Theorem}[section]
\newtheorem{proposition}[theorem]{Proposition}
\newtheorem{lemma}[theorem]{Lemma}
\newtheorem{corollary}[theorem]{Corollary}

\theoremstyle{definition}
\newtheorem{definition}[theorem]{Definition}

\theoremstyle{remark}
\newtheorem{remark}[theorem]{Remark}
\newtheorem{example}[theorem]{Example}

\newcommand{\CC}{\mathbb{C}}

\newcommand{\ZZ}{\mathbb{Z}}
\newcommand{\im}{\operatorname{im}}

\newcommand{\src}{\bm{s}}
\newcommand{\ran}{\bm{r}}

\newcommand{\supp}{\operatorname{supp}}

\newcommand{\Iso}{\operatorname{Iso}}
\newcommand{\Bisc}{\operatorname{Bis}_c}

\title[Tensor products of Steinberg algebras]{Tensor products of Steinberg algebras}

\author[S. W. Rigby]{Simon W. Rigby}

\address{Department of Mathematics: Algebra and Geometry, Ghent University, Belgium.}

\email{simon.rigby@ugent.be}

\begin{document}

\begin{abstract}
		We prove that $A_R(G) \otimes_R A_R(H) \cong A_R(G \times H)$, if $G$ and $H$ are Hausdorff ample groupoids. As part of the proof, we give a new universal property of Steinberg algebras. We then consider the isomorphism problem for tensor products of Leavitt algebras, and show that no diagonal-preserving isomorphism exists between $L_{2,R} \otimes L_{3,R}$ and  
	$L_{2,R} \otimes L_{2,R}$. Indeed, there are no unexpected diagonal-preserving isomorphisms between tensor products of finitely many Leavitt algebras. We give an easy proof that every $*$-isomorphism of Steinberg algebras over the integers preserves the diagonal, and it follows that $L_{2,\mathbb{Z}} \otimes L_{3,\mathbb{Z}} \not \cong L_{2,\mathbb{Z}} \otimes L_{2,\mathbb{Z}}$ (as $*$-rings).
\end{abstract}



\maketitle

\section{Introduction}

Steinberg algebras, introduced in \cite{steinberg0} and~\cite{clark2}, are convolution $R$-algebras defined over a locally compact, \'etale groupoid with totally disconnected, Hausdorff unit space (known as an ample groupoid).
When the input is Paterson's \cite{paterson} universal groupoid of an inverse semigroup $S$, the output is the inverse semigroup algebra $RS$. When the input is the boundary path groupoid \cite{kumjian1997graphs} of a graph $E$, the output is the Leavitt path algebra $L_R(E)$.  Groupoid techniques have had a strong influence on these subjects in recent years. For instance, the structure theory and ideal theory of Steinberg algebras is highly applicable to Leavitt path and inverse semigroup algebras \cite{clark2016ideals,orloff2016using,steinberg-simplicity, steinberg-primeness}. Steinberg algebras have also led to new insights and progress in $C^*$-algebras, especially on the topic of simple groupoid $C^*$-algebras (see \cite{brown2014simplicity, brown2017purely,clark2018simplicity}). Besides those already mentioned, there are interesting examples of Steinberg algebras that arise out of higher-rank graphs, self-similar graphs, and various kinds of dynamical systems \cite{beuter2017interplay,brown2014simplicity, clark2018generalized,clark2018simplicity,clark2017kumjian}.

Two very basic constructions, the {disjoint union} and the {product}, can be used to combine a pair of groupoids into a new groupoid.
It is intuitive and easy  to prove that the Steinberg algebra $A_R(G \sqcup H)$ is isomorphic to the direct sum of two ideals, $A_R(G) \oplus A_R(H)$, for any pair of ample groupoids $G$ and~$H$. On the other hand, there are examples to suggest that taking products of groupoids has the effect that \begin{equation} \label{main} A_R(G \times H) \cong A_R(G)\otimes_R A_R(H).\end{equation} An analogous result for groupoid $C^*$-algebras appears in \cite[Lemma 2.10]{austin2018groupoid}. When $G$ and $H$ are discrete, or when $G$ is discrete and principal, then (\ref{main}) is a straightforward calculation with bases or matrix units respectively. However, in the absence of some special property of the groupoids (for example, being discrete, or minimal and effective) it is more challenging to verify (\ref{main}). For this, it is useful to have a slightly more versatile universal property for Steinberg algebras than the one provided in \cite[Proposition~2.3]{clark}. We give such a universal property in Theorem~\ref{univ}.

Recently, there have been some investigations of constructions that produce a new groupoid from a groupoid $G$ and a group $\Gamma$. In~\cite[Theorem 3.4]{AHLS}, it is shown that the Steinberg algebra of the skew product groupoid of $G$ by $\Gamma$ (where $G$ is graded by $\Gamma$) is graded isomorphic to the Cohen-Montgomery {smash product} ring $A_R(G)\# \Gamma$. In \cite[Proposition 4.2]{homology}, it is shown that the Steinberg algebra of the semidirect product groupoid of $G$ and $\Gamma$ (where $\Gamma$ acts on $G$) is graded isomorphic to a partial skew group ring $A_R(G) \rtimes \Gamma$. Graded matrix rings with entries in a Steinberg algebra are also shown in \cite[\S4.2]{homology} to be graded isomorphic to the Steinberg algebras of certain graded groupoids constructed from $G$ and $\Gamma$. Our present investigation into the Steinberg algebras of product groupoids naturally falls into this theme, and it has some interesting applications.

The classical Leavitt algebra $L_{n,R}$ ($2 \le n < \infty$) is the universal $R$-algebra $A$ satisfying $A^n  \cong A$  as $A$-modules and $A^m \not \cong A$ for any $2 \le m < n$. When  interest in these algebras was revived in the 2000s, after the discovery of Leavitt path algebras in \cite{abrams2005leavitt} and \cite{ara2007nonstable},  there was an open question for several years to determine whether $L_{2,R} \otimes L_{2,R} \cong L_{2,R}$. This was motivated by Elliot's Theorem that says, among other things, that $\mathcal{O}_2 \otimes \mathcal{O}_2 \cong \mathcal{O}_2$, where $\mathcal{O}_2$ is the Cuntz $C^*$-algebra generated by two partial isometries. We now know, from Ara and Corti\~nas's calculation of the Hochschild homology \cite{AC},  that no ring isomorphism exists between $L_{2,K} \otimes L_{2,K}$ and $L_{2,K}$ for any field $K$. Some alternative proofs have also been given, using different homological arguments (see \cite[\S3.5]{abrams} and \cite[\S6.4]{LPAbook}). Furthermore, by the results of Brownlowe and S{\o}rensen \cite{BS}, there is no $*$-ring monomorphism $L_{2,\ZZ} \otimes L_{2,\ZZ} \hookrightarrow L_{2,\ZZ}$. However, it is still unknown if there exists a ring monomorphism $L_{2,R} \otimes L_{2,R} \hookrightarrow L_{2,R}$ for some ring of scalars $R$.

In \cite[\S6.5]{abrams} and \cite[Question 7.3.4.]{LPAbook}, a similar question was posed: Is $L_{2,R}\otimes L_{3,R} \cong L_{2,R} \otimes L_{2,R}$? While we are not able to answer this question completely, we do rule out the possibility that there is an isomorphism that preserves the diagonal subalgebras. That is, if $R$ is indecomposable, then there is no isomorphism $L_{2,R}\otimes L_{3,R} \to L_{2,R}\otimes L_{2,R}$ that maps $D_{2,R} \otimes D_{3,R}$ into $D_{2,R} \otimes D_{2,R}$ (or whose inverse maps $D_{2,R} \otimes D_{2,R}$ into $D_{2,R} \otimes D_{3,R}$). 
More generally, we show in Proposition \ref{no-diag-iso} that if there exists a diagonal-preserving isomorphism $\bigotimes_{i = 1}^p L_{n_i,R} \to \bigotimes_{i = 1}^q L_{m_i,R}$, then $p = q$ and $(n_1, \dots, n_p) = (m_{\sigma(1)}, \dots, m_{\sigma(p)})$ for some permutation $\sigma \in S_p$. The proof uses important theorems of Matui \cite[Theorem 5.12]{matui} on products of shifts of finite type, and Steinberg \cite[Corollary 5.8]{steinberg-new} on groupoid reconstruction (for which \cite[Corollary 3.2]{carlsen2017diagonal} or \cite[Theorem 3.1]{ara2017reconstruction} could also have been used).
 In~\cite{carlsen}, Carlsen proved that every $*$-isomorphism of integral Leavitt path algebras preserves the diagonal. We show that the same holds for integral Steinberg algebras. This leads to Corollary \ref{corr 3}, that $\bigotimes_{i = 1}^p L_{n_i, \ZZ}$ is  $*$-isomorphic to $\bigotimes_{j = 1}^q L_{m_j,\ZZ}$ if and only if $p = q$ and $(n_1, \dots, n_p) = (m_{\sigma(1)}, \dots, m_{\sigma(p)})$ for some permutation $\sigma \in S_p$.

The paper is organised in six sections. In \S\ref{s: prelim}, we set up our conventions and terminology for groupoids, and give the definition of a Steinberg algebra. In \S\ref{s: univ prop}, we establish a new universal property for Steinberg algebras. In \S\ref{s: tensors}, we prove that $A_R(G) \otimes A_R(H)\cong A_R(G \times H)$ for any pair of Hausdorff ample groupoids. In \S\ref{s: stars}, we show that every projection in an integral Steinberg algebra belongs to the diagonal subalgebra. Consequently, if there is a $*$-isomorphism $A_\ZZ(G) \cong A_\ZZ(H)$ and either $G$ or $H$ is effective, then $G \cong H$. In \S\ref{s: leavitts}, we make some progress on the isomorphism problem for tensor products of Leavitt algebras.
\section{Preliminaries} \label{s: prelim}

A groupoid is a small category in which every morphism is invertible. For a groupoid $G$, we use the following conventions: the \textit{unit space} is $G^{(0)} = \{xx^{-1} \mid x \in G \} = \{x^{-1}x \mid x \in G\}$, the \textit{source} map is $\src: G \to G^{(0)}$, $\src(x) = x^{-1} x$, the \textit{range} map is $\ran: G \to G^{(0)}$, $\ran(x) = xx^{-1}$,  and the set of \textit{composable pairs} is $G^{(2)} = \{(x,y) \in G \times G \mid \src(x) = \ran(y) \}$. We use the notation $G_u = \src^{-1}(u)$, $G^v = \ran^{-1}(v)$, and $G_u^v = G_u \cap G^v$, if $u,v \in G^{(0)}$ are units. The group $G_u^u$ is called the \textit{isotropy group} based at~$u$. The \textit{isotropy subgroupoid} of $G$ is $\Iso(G) = \{x \in G \mid \src(x) = \ran(x) \} = \bigcup_{u \in G^{(0)}} G_u^u$. We say that a groupoid $G$ is \textit{transitive} if $G_u^v \ne \emptyset$ for every pair of units $u, v \in G^{(0)}$, and $G$ is \textit{principal} if $\Iso(G) = G^{(0)}$.

A topological groupoid is a groupoid in which the multiplication and inversion maps are continuous (and therefore the source and range maps are continuous too). An \textit{ample groupoid} is a topological groupoid $G$ in which $\src$ is a local homeomorphism and $G^{(0)}$ is locally compact, totally disconnected, and Hausdorff. Equivalently, an ample groupoid is a topological groupoid $G$ in which $\src$ is an open map and the topology on $G$ has a basis of \textit{compact open bisections}; that is, compact open subsets of $G$ on which the restrictions of $\src$ and $\ran$ are injective. We say that an ample groupoid is \textit{effective} if the interior of the isotropy subgroupoid is the unit space; that is, if $\Iso(G)^\circ = G^{(0)}$. 

The \textit{product} of two groupoids $G_1 \times G_2$ has unit space $G_1^{(0)}\times G_2^{(0)}$, and the following structure maps. For all $(x_1, y_1) \in G_1^{(2)}$ and $(x_2, y_2) \in G_2^{(2)}$,
\begin{align*}
\src(x_1, x_2) &= (\src(x_1),\src(x_2)), & (x_1, x_2)^{-1} &= (x_1^{-1}, x_2^{-1}),\\ \ran(x_1,x_2) &= (\ran(x_1),\ran(x_2)),
 & (x_1, x_2)(y_1, y_2) &= (x_1y_1,x_2 y_2).
\end{align*}
If $G_1$ and $G_2$ are topological groupoids, then $G_1 \times G_2$ equipped with the product topology is again a topological groupoid. The product of two ample groupoids is again an ample groupoid.

Unless specified otherwise, the assumption is that $R$ is an arbitrary commutative ring with~1. All tensor products in this paper are over $R$. If $A$ and $B$ are $*$-algebras over $R$, then $A \otimes B$ is a $*$-algebra with $(\sum a_i \otimes b_i)^* = \sum a_i^* \otimes b_i^*$.

 In an ample groupoid $G$, let $\Bisc(G)$ be the set of compact open bisections in $G$, which is an inverse semigroup under the product and inverse operations: 
	\begin{align*}
	BC = \{xy \mid x \in B, y \in C, \ran(y) = \src(x)\}, &&	B^{-1} = \{x^{-1} \mid x \in B\}.	
	\end{align*}
	The set of compact open subsets of $G^{(0)}$ is denoted by $\Bisc(G^{(0)})$. It is the set of idempotents in the semigroup $\Bisc(G)$. For any $B \in \Bisc(G)$, let $\bm{1}_B: G \to R$ be the characteristic function of $B$. 
 
\begin{definition}[Steinberg algebras] \cite{steinberg0} \label{st-def}
Let $G$ be an ample groupoid. The \textit{Steinberg algebra of $G$ over~$R$}, denoted by $A_R(G)$, is the $R$-submodule of $R^G$ generated by the characteristic functions
\[
\{\bm{1}_B \mid B \in \Bisc(G)\},
\]
and equipped with the \textit{convolution product}:
\begin{align} \label{convolve}
f*g(x) = \sum_{y \in G_{\src(x)}} f(xy^{-1})g(y) = \sum_{\substack{(z,y) \in G^{(2)},\\ zy = x}} f(z)g(y),
\end{align}
for all $f, g \in A_R(G)$ and $x \in G$.
The \textit{diagonal subalgebra} of $A_R(G)$ is the subalgebra $D_R(G)$ generated by the commuting idempotents:
\[
\{\bm{1}_U \mid U \in \Bisc(G^{(0)})\}.
\]
\end{definition}
If $A, B \in \Bisc(G)$ then (\ref{convolve}) yields $\bm{1}_A * \bm{1}_B = \bm{1}_{AB}$. In general, \begin{align} \notag
D_R(G) &= \{f \in A_R(G) \mid \supp(f) \subseteq G^{(0)}\} \cong A_R(G^{(0)}), \\
 \label{lccs}
A_R(G) &\supseteq \{f: G \to R \mid f \text{ is locally constant, compactly supported}\}. 	
 \end{align}
If $G$ is effective, then $D_R(G)$ is a maximal commutative subalgebra of $A_R(G)$ \cite[Corollary~2.4]{steinberg-new}.  If $G$ is Hausdorff, then (\ref{lccs}) is an equality. The Steinberg algebra $A_R(G)$ is unital if and only if $G^{(0)}$ is compact \cite[Proposition 4.11]{steinberg0}. However, $A_R(G)$ is always a locally unital ring (every finitely generated subalgebra is unital).
If $R$ is a commutative unital ring with an involution $\overline{\phantom{f}}: R \to R$, then $A_R(G)$ is a $*$-algebra, with the involution $^*: A_R(G) \to A_R(G)$ given by:
\begin{align} \label{invol}
f^*(x) = \overline{f(x^{-1})} &&\text{ for all } f \in A_R(G), \ x \in G.
\end{align}

We briefly discuss graded algebras and graded groupoids, because it is worthwhile to make one or two comments about these later.	Let $\Gamma$ be a group. We say that an $R$-algebra  $A$ is \emph{$\Gamma$-graded} if it decomposes into $R$-submodules $A = \bigoplus_{\gamma \in \Gamma} A_\gamma$, with the property that $A_{\gamma_1} A_{\gamma_2} = A_{\gamma_1 \gamma_2}$ for all $\gamma_1, \gamma_2 \in \Gamma$. The $R$-submodules $A_\gamma$ are called \emph{homogeneous components}. A homomorphism of $\Gamma$-graded $R$-modules $\varphi: A \to B$ is a \emph{$\Gamma$-graded homomorphism} if $\varphi(A_\gamma) \subseteq \varphi(B_\gamma)$ for all $\gamma \in \Gamma$.

We say that a topological groupoid $G$ is \emph{$\Gamma$-graded} if there is a homomorphism $c: G \to \Gamma$ that is continuous with respect to the discrete topology on $\Gamma$. The map $c$ is called a \emph{cocycle}. The clopen sets $c^{-1}(\gamma)$ are called \emph{homogeneous components}, and the collection $\{c^{-1}(\gamma) \mid \gamma \in \Gamma\}$ is a partition of $G$ with the property that $c^{-1}(\gamma)c^{-1}(\delta) = c^{-1}(\gamma \delta)$ for all $\gamma, \delta \in \Gamma$.

\begin{example}  \label{gr-example} Let $\Gamma$ and $\Delta$ be groups.
\begin{enumerate} \item \cite[Lemma 3.1]{clark3}
		If $G$ is a $\Gamma$-graded ample groupoid, then $A_R(G)$ is a $\Gamma$-graded algebra, with the homogeneous components: \[A_R(G)_\gamma = \{f \in A_R(G) \mid f^{-1}(R \setminus \{0\}) \subseteq c^{-1}(\gamma)\}.\]
 \item \cite[Example 1.16]{hazrat2016graded} If $A$ is a $\Gamma$-graded $R$-algebra and $B$ is a $\Delta$-graded $R$-algebra, then $A \otimes B$ is a $(\Gamma \times \Delta)$-graded $R$-algebra, with the homogeneous components: \[(A \otimes B)_{(\gamma, \delta)} = A_\gamma \otimes B_{\delta}.\]	
 If $\Gamma = \Delta$ is abelian, then the group homomorphism $\Gamma \times \Gamma \to \Gamma$ sending $(\gamma_1, \gamma_2)$ to $\gamma_1 \gamma_2$ induces a quotient grading on $A \otimes B$, which makes it a $\Gamma$-graded algebra with the homogeneous components: \[(A \otimes B)_\gamma = \sum_{\mu \in \Gamma} A_{\gamma \mu} \otimes A_{\mu^{-1}}.\]
 \item If $G$ and $H$ are topological groupoids graded by continuous cocycles $c: G \to \Gamma$ and $c: H \to \Delta$, then $G \times H$ is a $(\Gamma \times \Delta)$-graded topological groupoid, as determined by the cocycle:
 	\begin{align*}
 	c: G \times H \to \Gamma	 \times \Delta, && c(x,y) = (c(x),c(y)).
 	\end{align*}
 	If $\Gamma = \Delta$ is abelian, then the group homomorphism $\Gamma \times \Gamma \to \Gamma$ sending $(\gamma_1, \gamma_2)$ to $\gamma_1 \gamma_2$ induces a quotient grading on $G \times H$, which is determined by the cocycle:
 	\begin{align*}
 	c: G \times H \to \Gamma	, && c(x,y) = c(x)c(y).
 	\end{align*}
\end{enumerate}
\end{example}


\section{A universal property} \label{s: univ prop}

In this section, we give a modified version of \cite[Proposition 2.3]{clark}, where it was shown that Steinberg algebras are universal algebras generated by {representations} of the set of homogeneous compact open bisections. We show that $A_R(G)$ is universal for representations of any ``suitable" set $\mathcal{B}$ of compact open bisections. The conditions we impose on $\mathcal{B}$ are satisfied by some useful bases in well-known ample groupoids (see Example \ref{ex}).

\begin{definition} \cite{clark2}
	Let $G$ be a Hausdorff ample groupoid, and let $\mathcal{B}$ be a subsemigroup of $\Bisc(G)$. Let $Q$ be an $R$-algebra. A \textit{representation} of $\mathcal{B}$ in $Q$ is a family
	\[
	\{t_B \mid B \in \mathcal{B} \} \subseteq Q
	\]
	satisfying the following relations:
	\begin{enumerate}[\rm (R1)]
		\item $t_\emptyset = 0$ if $\emptyset \in \mathcal{B}$;
		\item $t_A t_B = t_{AB}$ for all $A,B \in \mathcal{B}$;
		\item $\sum_{B \in F} t_B = t_{\bigcup F}$ for all finite $F \subseteq \mathcal{B}$ such that $\bigcup F \in \mathcal{B}$ and $B \cap B' = \emptyset$ for distinct $B, B' \in F$.\\
	\end{enumerate}
\end{definition}

\begin{remark} \label{remark-clark}
In \cite[Theorem 3.10]{clark2} and \cite[Theorem 3.11]{steinberg-1}, it is shown that when $\mathcal{B} = \Bisc(G)$ then (R3) is equivalent to the easier-to-check condition: $t_{U \cup U'} = t_U + t_{U'}$ for all disjoint pairs $U, U' \in \Bisc(G^{(0)})$. For more general $\mathcal{B}$, however, we need the full strength of (R3) as stated above.
\end{remark}

Here is our universal property for Steinberg algebras, which extends \cite[Proposition 2.3]{clark} and its predecessors \cite[Theorem 3.11]{steinberg-1} and \cite[Theorem~3.10]{clark2}.

\begin{theorem} \label{univ}
	Let $G$ be a Hausdorff ample groupoid. Suppose $\mathcal{B}$ is a subsemigroup of $\Bisc(G)$, and a basis for the topology on $G$, and also satisfies the properties:
		\begin{itemize}
			\item [\rm (Int)] $\emptyset\in \mathcal{B}$ and $\mathcal{B}$ is closed under finite intersections;
			\item [\rm (RC)] For all $B, C \in \mathcal{B}$, the relative complement $B\setminus C$ is a finite union of disjoint sets in $\mathcal{B}$. 
		\end{itemize}
	Then $\{\bm{1}_B \mid B \in \mathcal{B}\}$ is a representation of $\mathcal{B}$ in $A_R(G)$, and it generates $A_R(G)$ as an $R$-module. Moreover, every $R$-algebra $Q$ containing a representation $\{t_B \mid B \in \mathcal{B}\}$ admits a unique homomorphism $\pi: A_R(G) \to Q$ such that $\pi(\bm{1}_B) = t_B$ for all $B \in \mathcal{B}$.
\end{theorem}

\begin{proof}
	Clearly the set $\{\bm{1}_B \mid B \in \mathcal{B}\}$ satisfies (R1) and (R3). Since $\bm{1}_B * \bm{1}_C = \bm{1}_{BC}$, it also satisfies (R2). Therefore, it is a representation of $\mathcal{B}$ in $A_R(G)$. 
	
	\textit{Claim 1: } We claim that for each $f \in A_R(G)$, there exists a finite set $F \subseteq \mathcal{B}$ whose elements are mutually disjoint, such that $f$ is a linear combination of the set $\{\bm{1}_B \mid B \in F\}$. By \cite[Corollary~1.14]{rigby}, there is an expression $f = \sum_{i = 1}^n s_i \bm{1}_{D_i}$ where $s_1, \dots, s_n \in R \setminus \{0\}$ and $D_1, \dots, D_n \in \mathcal{B}$. Assume that no two $D_i$ in the expression are the same. For each $s \in \im(f) \setminus \{0\}$ we have:
	\begin{align*} 
	f^{-1}(s) &= \bigcup_{\substack{I \subseteq \{1, \dots, n\} \\ s = \sum_{i \in I}{s_i}}}\  B_I, & \text{where } \qquad B_I &= \bigcap_{\substack{i\in I \\ j \notin I}} D_i \setminus D_j.
	\end{align*}
	It follows from our assumptions that each nonempty $B_I$ is a finite union of disjoint members of $\mathcal{B}$; say $B_I = \bigcup{F_I}$ where $F_I \subset \mathcal{B}$ is a set of mutually disjoint basic open sets. Moreover, if $I, J \subseteq \{1, \dots, n\}$ and $I \ne J$, then $B_I \cap B_J = \emptyset$. To complete the proof of the claim, note that 
	\[
	f = \sum_{s \in \im f \setminus \{0\}} s \bm{1}_{f^{-1}(s)} = \sum_{s \in \im f \setminus \{0\}} \sum_{\substack{I \subseteq \{1, \dots, n\} \\ s = \sum_{i \in I}{s_i}}} s \bm{1}_{B_I} 
	 = \sum_{s \in \im f \setminus \{0\}} \sum_{\substack{I \subseteq \{1, \dots, n\} \\ s = \sum_{i \in I}{s_i}}} \sum_{B \in F_I} s \bm{1}_{B}.
	\]
	
	The next claim	follows the same steps as \cite[Theorem~3.10]{clark2}, but we write out the proof for ease of reference.
	
	\textit{Claim 2: }  Let $Q$ be an $R$-algebra containing a representation $\{t_B \mid B \in \mathcal{B}\}$ of $\mathcal{B}$.  We claim that the following mapping is well-defined:
	\begin{align} \label{pi-map}
	\pi: A_R(G) \to Q, && \sum_{B \in F} r_B \bm{1}_B \mapsto \sum_{B \in F} r_B t_B
	\end{align}
	for all finite $F \subseteq \mathcal{B}$ whose elements are mutually disjoint, and all scalars $r_B \in R \setminus \{0\}$. From the first claim, each element of $A_R(G)$ has at least one expression $f = \sum_{B \in F} r_B \bm{1}_B$. Now suppose $f \in A_R(G)$ has two such expressions:
	\begin{equation*}
	\sum_{B \in F} r_B \bm{1}_B = f = \sum_{B' \in F'} s_{B'} \bm{1}_{B'}, 
	\end{equation*}
	where	$F, F' \subseteq \mathcal{B}$, and $A,B \in F$ or $A, B \in F'$ implies  $A \cap B = \emptyset$, and $r_B, s_{B'} \in R$ for all $B \in F$, $B' \in F'$.
	It is necessary to show that
	\begin{equation*}
		\sum_{B \in F} r_B t_B = \sum_{B' \in F'} s_{B'} t_{B'}.
	\end{equation*} 
	As in \cite{clark2}, let $K = \{B \cap B' \mid B \in F,\ B' \in F',\ B \cap B' \ne \emptyset\}$ and notice that $K \subseteq \mathcal{B}$. If $B \in F$ then $B = \bigsqcup \{W \in K \mid W \subseteq B\}$. Applying (R3), we get that $t_B = \sum_{W \in K, W \subseteq B} t_W$ for every $B \in F$. Similarly, $t_{B'} = \sum_{W \in K, W \subseteq B'}t_W$ for every $B' \in F'$. Therefore,
	\begin{align*}
	 \sum_{B \in F} r_B t_B = \sum_{B \in F} \sum_{\substack{W \in K\\ W \subseteq B}} r_B t_W &= \sum_{W \in K} \Bigg(\sum_{\substack{B \in F\\ W \subseteq B}} r_B t_W\Bigg), \text{ and }  \\
	 \sum_{B' \in F'} s_{B'} t_{B'} =  \sum_{B' \in F'} \sum_{\substack{W \in K\\ W \subseteq B'}}s_{B'}t_W &=\sum_{W \in K} \Bigg(\sum_{\substack{B' \in F' \\ W \subseteq B'}} s_{B'} t_W\Bigg).
	\end{align*}
	All that remains is to show that $\sum_{B \in F, W \subseteq B} r_B t_W = \sum_{B' \in F', W \subseteq B'} s_{B'}t_W$ for all $W \in K$.
	To do so, notice that $f$ is constant and equal to $\sum_{B\in F, W \subseteq B}r_B = \sum_{B' \in F', W \subseteq B'}s_{B'}$ on $W$. We conclude that (\ref{pi-map})  defines the function $\pi$ unambiguously.
	
	\textit{Claim 3:} We claim that $\pi$ is a homomorphism. Let $f, g \in A_R(G)$. Write $f$ and $g$ as linear combinations of characteristic functions of sets in $\mathcal{B}$. Using the fact that $\bm{1}_B * \bm{1}_C = \bm{1}_{BC}$ for all $B, C \in \mathcal{B}$, and $t_B t_C = t_{BC}$ for all $B, C \in \mathcal{B}$ by (R2), we can conclude that $\pi(f*g) = \pi(f)\pi(g)$.
	
	The conclusion is that $A_R(G)$ enjoys the universal property in the statement of the theorem.
\end{proof}

\begin{example}  \label{ex}
	Here are a few examples of Hausdorff ample groupoids $G$ having a special subsemigroup $\mathcal{B} \subseteq \Bisc(G)$ that forms the basis for the topology on $G$ and has the properties (Int) and (RC).
	\begin{enumerate}
		\item	If $G$ is graded by a group, let $\mathcal{B}$ be the set of all homogeneous compact open bisections.
		\item If $G$ and ${H}$ are two ample groupoids, let $\mathcal{B}$ be the basis for $G \times {H}$ consisting of products of compact open bisections $A \times B$, where $A \in \Bisc(G)$ and $B \in \Bisc(H)$. This example is an important one that is revisited in Lemma \ref{K-lemma}.
		\item If $G_E$ is the boundary path groupoid of a directed graph $E$, let $\mathcal{B}$ be the basis of sets \begin{align*} \label{basic sets def}
		{Z}(\alpha, \beta) &= \big\{(\alpha x, |\alpha| - |\beta|, \beta x) \mid x \in r(\alpha)\partial E \big\},\\
		{Z}(\alpha, \beta, F) &= {Z}(\alpha, \beta) \setminus \bigcup_{e \in F} {Z}(\alpha e, \beta e),
		\end{align*}
		where $\alpha$ and $\beta$ are arbitrary finite paths with a common range, and $F$ is an arbitrary finite subset of edges emitted by the range of $\alpha$. By \cite[Lemma~2.15]{rigby}, $\mathcal{B}$ is closed under finite intersections. To prove that $\mathcal{B}$ has property (RC),  we can dismiss the trivial cases where $Z( \alpha, \beta, F) \setminus Z(\gamma, \delta, H)$ is equal to $\emptyset$ or $Z(\alpha, \beta, F)$.
		 Otherwise, by \cite[Lemma~2.15]{rigby}, $Z(\alpha, \beta,F) \setminus Z(\gamma, \delta, H)$ is equal to one of the following:
		 \begin{align*} Z(\alpha, \beta, F) \setminus Z(\alpha, \beta, F \cup H) =\ & \bigsqcup_{e \in H \setminus F} Z(\alpha e, \beta e), \text{ or}\\
		Z(\alpha, \beta, F) \setminus Z(\alpha \kappa, \beta \kappa, H)  =\ & Z(\alpha, \beta, F \cup \{\kappa_1\}) \sqcup Z(\alpha \kappa_1, \beta \kappa_1, \{\kappa_2\}) \sqcup \dots \\ & \sqcup Z(\alpha \kappa_1 \dots \kappa_{|\kappa|-1}, \beta \kappa_1 \dots \kappa_{|\kappa|-1}, \{\kappa_{|\kappa|}\}) \\
		& \sqcup \bigsqcup_{e \in H} Z(\alpha \kappa e, \beta \kappa e). 
		\end{align*}
	\end{enumerate}
\end{example}

\section{Tensor products of Steinberg algebras} \label{s: tensors}

We first give an easy application of the universal property. A version of Proposition \ref{scalars} for Leavitt path algebras appears in \cite[Corollary~1.5.14]{LPAbook}.

\begin{proposition} \label{scalars}
	Let $S$ be a commutative unital $R$-algebra. If $G$ is any Hausdorff ample groupoid, then $S \otimes_R A_R(G) \cong A_S(G)$ as $S$-algebras.
\end{proposition}

\begin{proof}
	The set $\{1 \otimes \bm{1}_B \mid  B \in \Bisc(G)\}$ is a representation of $\Bisc(G)$ inside $S \otimes A_R(G)$. By Theorem~\ref{univ}, there is a unique $S$-algebra homomorphism $\pi: A_S(G) \to S \otimes A_R(G)$ such that $\pi(\bm{1}_B) = 1 \otimes \bm{1}_B$ for all $B \in \Bisc(G)$. On the other hand, the universal property of tensor products gives the inverse map $\sigma: S \otimes A_R(G) \to A_S(G)$, defined on simple tensors by $s \otimes f \mapsto sf$.
\end{proof}

In order to apply Theorem~\ref{univ} to products of ample groupoids, a technical lemma is needed.
\begin{lemma} \label{K-lemma}
	Let $G$ and $H$ be  Hausdorff ample groupoids. Define
	\begin{align*}
	\mathcal{K} &= \{A \times B \mid A \in \Bisc(G), B \in \Bisc({H})\}.
	\end{align*}
	 Then $\mathcal{K}$ is an inverse subsemigroup of $\Bisc(G \times H)$ and a basis for the topology on $G \times H$, such that $\mathcal{K}$ is closed under finite intersections, and for every $K,L \in \mathcal{K}$, the relative complement $K \setminus L$ is a finite union of disjoint sets in $\mathcal{K}$.
	For all $A \times B \in \mathcal{K}$, define \[t_{A \times B} = \bm{1}_A \otimes \bm{1}_B \in A_R(G)\otimes A_R({H}).\]
	Then $\{t_K \mid K \in \mathcal{K}\}$ is a representation of $\mathcal{K}$ in $A_R(G) \otimes A_R(H)$.
\end{lemma}

\begin{proof}
	Obviously $\mathcal{K}$ is a basis for the topology on $G \times H$. For all $A_1 \times B_1, A_2 \times B_2 \in \mathcal{K}$:
	\begin{align*}
	(A_1 \times B_1)(A_2 \times B_2) = &A_1 A_2 \times B_1 B_2, \\ (A_1\times B_1)^{-1} &= A_1^{-1}\times B_1^{-1},\\ (A_1 \times B_1) \cap (A_2 \times B_2) = &(A_1 \cap A_2) \times (B_1 \cap B_2),\\
	(A_1 \times B_1) \setminus (A_2 \times B_2)
	= &\big((A_1 \setminus A_2) \times (B_1 \cap B_2)\big) \sqcup \big((A_1 \cap A_2) \times (B_1 \setminus B_2)\big) \\ & \sqcup \big((A_1 \setminus A_2)\times (B_1 \setminus B_2)\big).
	\end{align*}
	 This proves that $\mathcal{K}$ is an inverse subsemigroup  of $\Bisc(G \times H)$, with the properties (Int) and (RC).
	  
	  To prove that $\{t_K \mid K \in \mathcal{K}\}$ is a representation of $\mathcal{K}$ in $A_R(G)\otimes A_R(H)$, the first two conditions are easy to verify: \begin{enumerate}[(R1)]
	  	\item $t_\emptyset = t_{\emptyset \times \emptyset} = \bm{1}_\emptyset \otimes \bm{1}_{\emptyset} =  0$.
	  	\item For all $A \times B$, $C \times D \in \mathcal{K}$, \[t_{A \times B} t_{C \times D} = (\bm{1}_A \otimes \bm{1}_B)(\bm{1}_C \otimes \bm{1}_D) = \bm{1}_A \bm{1}_C \otimes \bm{1}_B \bm{1}_D =   \bm{1}_{AC}\otimes \bm{1}_{BD} = t_{AC \times BD} = t_{(A \times B)(C \times D)}.\]
	  \end{enumerate}
	  It remains to verify (R3). Let $F  \subseteq \mathcal{K}$ be a finite set such that $\bigcup F \in \mathcal{K}$ and  $(A \times B) \cap (A'\times B') = \emptyset$ whenever $A \times B$ and $A'\times B'$ are distinct members of $F$. We need to show that 
	  \begin{equation*}
		 \sum_{K \in F} t_K = t_{\bigcup F}.
	  \end{equation*}
	  To do so, let $P = \{A \mid A \times B \in F\}$ and $Q = \{B \mid A \times B \in F\}$. The sets that are elements of $P$ (or $Q$) are not necessarily disjoint, so we have to do some calculations that involve making them disjoint and then doing a reconciliation later.
	  Let
	  \begin{align*}
		  X = \Bigg\{ \bigcap_{A \in P, x \in A}A \ \ \Big|\ \ x \in \bigcup P\Bigg\},&&Y = \Bigg\{ \bigcap_{B \in Q, y \in B}B \ \ \Big|\ \ y \in \bigcup Q\Bigg\}.
	\end{align*}
	Both these sets $X$ and $Y$ are finite, because $P$ and $Q$ are finite, and $\bigcup X = \bigcup P$ while $\bigcup Y = \bigcup Q$. Also note that the sets that are elements of $X$ are mutually disjoint, and the same for $Y$. For all $A \in P$, we have that $A = \bigsqcup\{C \in X: C \subseteq A\}$ and for all $B \in Q$ we have $B = \bigsqcup\{D \in Y : D \subseteq B\}$.
	
	We claim that $\bigcup F = \bigcup X \times \bigcup Y$. Indeed, if $(x,y) \in \bigcup F$ then $(x,y) \in A \times B$ for some $A \times B \in F$. Then it is easy to see that $x \in \bigcap_{A \in P, x \in A}A \subseteq \bigcup X $ and $y \in \bigcap_{B \in Q, y \in B}B \subseteq \bigcup Y$. Thus $\bigcup F \subseteq \bigcup X \times \bigcup Y$.  On the other hand, suppose $(x,y) \in (\bigcup X \times \bigcup Y) \setminus \bigcup F$. Since $\bigcup F \in \mathcal{K}$, there exist sets $T \in \Bisc(G)$ and $S \in \Bisc(H)$ such that $\bigcup F = T \times S$. Then, either $x \notin T$ or $y \notin S$. If $x \notin T$ then $x \notin A$ for any $A \times B \in F$, which is a contradiction because $x \in \bigcup P$ and $\bigcup P \subseteq T$. Similarly, $y \notin S$ reaches a contradiction. Therefore $\bigcup F = \bigcup X \times \bigcup Y$. 
	
	Now, $X$ and $Y$ are sets of mutually disjoint sets, so
	\begin{align} \label{comp1}
	t_{\bigcup F} = t_{\bigcup X \times \bigcup Y} &= \bm{1}_{\bigcup X} \otimes \bm{1}_{\bigcup Y} = \Big(\sum_{C \in X} \bm{1}_{C}\Big)\otimes\Big(\sum_{D \in Y} \bm{1}_D\Big) = \sum_{C \in X}\sum_{D \in Y}  \bm{1}_C \otimes \bm{1}_D. 
	\end{align}
	On the other hand, 
	\begin{align} \label{comp2} \notag
	\sum_{K \in F} t_K = \sum_{A\times B \in F} t_{A \times B} &= \sum_{A \times B \in F} \bm{1}_A \otimes \bm{1}_B = \sum_{A \times B \in F}\bm{1}_{\bigsqcup \{C \in X\mid C \subseteq A\}}\otimes\bm{1}_{\bigsqcup \{D \in Y\mid D \subseteq B\}}
	\\
	& = \sum_{A \times B \in F} \Big(\sum_{\substack{C \in X \\ C \subseteq A}} \bm{1}_C\Big) \otimes  \Big(\sum_{\substack{D \in Y \\ D \subseteq B}} \bm{1}_D\Big) = \sum_{A \times B \in F}\Big(\sum_{\substack{C \in X \\ C \subseteq A}} \sum_{\substack{D \in Y \\ D \subseteq B}} \bm{1}_{C}\otimes \bm{1}_D \Big).
	\end{align}
	Note that for $C \in X$ and $A \in P$,  $C \cap A \ne \emptyset$ implies $C \subseteq A$. Similarly, for $D \in Y$ and $B \in Q$,  $D \cap B \ne \emptyset$ implies $D \subseteq B$. Given a pair $C \in X$ and $D \in Y$, we have that $C \times D \subseteq \bigcup X \times \bigcup Y = \bigcup F$. It follows that $(C \times D) \cap (A \times B) \ne \emptyset$ for some $A \times B \in F$, and therefore $C \times D \subseteq A \times B$. Moreover, this set $A \times B$ is the unique element of $F$ that contains $C \times D$ as a subset, because all the elements of $F$ are mutually disjoint. We conclude that for each pair  $(C,D)$ with $C \in X$ and $D \in Y$ there exists a unique $A \times B \in F$ such that $C \subseteq A$ and $D \subseteq B$. Therefore, continuing from (\ref{comp1}) and (\ref{comp2}) we prove that (R3) holds:
	\[
	\sum_{K \in F} t_K = \sum_{A \times B \in F}\Big(\sum_{\substack{C \in X \\ C \subseteq A}} \sum_{\substack{D \in Y \\ D \subseteq B}} \bm{1}_{C}\otimes \bm{1}_D \Big) = \sum_{C \in X} \sum_{D \in Y} \bm{1}_C \otimes \bm{1}_D = t_{\bigcup F}.
	\]
\end{proof}

And now for the main result.

\begin{theorem} \label{tensor}
	Let $G$ and $H$ be Hausdorff ample groupoids. There is a $*$-isomorphism
	\begin{equation*}
	\sigma: A_R(G) \otimes A_R(H) \to A_R(G \times H) 	\end{equation*}
	such that $\sigma\big(D_R(G)\otimes D_R(H)\big) = D_R(G\times H)$.
\end{theorem}

\begin{proof}
	Define the map $\widetilde \sigma: A_R(G) \times A_R(H) \to A_R(G\times H)$ that sends $(f,g)$ to $\widetilde \sigma(f,g) \in A_R(G \times H)$ where $\widetilde \sigma (f,g)(x,y) = f(x)g(y)$. Clearly $\widetilde \sigma$ is bilinear, so it induces a linear map:
	\begin{align*}
	\sigma: A_R(G)\otimes A_R(H) \to A_R(G \times H), && \sigma (f \otimes g) (x,y) = f(x)g(y).
	\end{align*}
	If $A \in \Bisc(G)$ and $B \in \Bisc(H)$ then it is clear that $\sigma(\bm{1}_A \otimes \bm{1}_B) = \bm{1}_{A \times B}$.
	By Lemma~\ref{K-lemma} and Theorem~\ref{univ}, there exists a unique $R$-algebra homomorphism $\pi: A_R(G \times H) \to A_R(G) \otimes A_R(H)$ such that $\pi(\bm{1}_{A \times B}) =  \bm{1}_A \otimes \bm{1}_B$ for all $A \times B \in \mathcal{K}$. Then $\pi$ and $\sigma$ are mutually inverse isomorphisms of $R$-algebras. It is clear that the isomorphisms are $*$-preserving (see (\ref{invol})) and diagonal-preserving.
\end{proof}


We briefly discuss the graded situation.
If $G$ is $\Gamma$-graded and $H$ is $\Delta$-graded, then $A_R(G \times H)$ is $(\Gamma \times \Delta)$-graded in the sense of Example~\ref{gr-example}~(1) and (3), and $A_R(G) \otimes A_R(H)$ is $(\Gamma \times \Delta)$-graded in the sense of Example~\ref{gr-example}~(2). In this case, the isomorphism from Theorem \ref{tensor} is a $(\Gamma \times \Delta)$-graded isomorphism. If $\Gamma = \Delta$ is abelian, then $G \times H$ and $A_R(G)\otimes A_R(H)$ can both be $\Gamma$-graded, as described in Example~\ref{gr-example}, and the isomorphism from Theorem \ref{tensor} is $\Gamma$-graded.

%
%
%
%

\section{$*$-isomorphisms of Steinberg algebras over $\ZZ$} \label{s: stars}

In this section, we are interested in the case where $R$ is a unital subring of $\mathbb{C}$ and the involution $\overline{\phantom{h}}$ is complex conjugation. Given a unital subring $R$ of $\mathbb{C}$ that is closed under complex conjugation, we shall call $R$ \textit{kind} if for all $\lambda_0, \dots, \lambda_n \in R$,
\[
\lambda_0 = |\lambda_0|^2 + \sum_{i = 1}^n |\lambda_i|^2 \ \text{ implies }\ \lambda_1 = \dots = \lambda_n = 0.
\]
The above definition is from \cite{carlsen}. For example, $\ZZ$ and $\ZZ[i]$ are kind. In \cite{carlsen} and \cite[Example 2.11]{JS}, there are some additional examples and non-examples of kind $*$-subrings of $\mathbb{C}$.

Recall from (\ref{invol}) that $A_R(G)$ is a $*$-algebra.  In a $*$-algebra $A$,  an element $a \in A$ is called a \textit{projection} if $a = a^2 = a^*$. A special case of Proposition \ref{kind} is proved for Leavitt path algebras in \cite[Proposition~4]{carlsen}. The proof here is surprisingly short, and it demonstrates the importance of kindness.

\begin{proposition}  \label{kind}
	Let $G$ be an ample groupoid and let $R$ be a kind, unital $*$-subring of $\mathbb{C}$. If $p \in A_R(G)$ is a projection, then $p \in D_R(G)$.
\end{proposition}

\begin{proof}
	Given that $p \in A_R(G)$ is a projection, we have $p = p*p = p^* = p*p^*$. For any unit $u \in G^{(0)}$,	\begin{equation*} 
	p(u) = p * p^*(u) = \sum_{y \in G_u}p(y^{-1})p^*(y) = \sum_{y \in G_u} p(y^{-1}) \overline{p(y^{-1})} = |p(u)|^2 + \sum_{ y \in G_u \setminus\{u\}} |p(y^{-1})|^2.
	\end{equation*}
	As usual, there are only finitely many $y \in G_u$ with $p(y^{-1}) \ne 0$. Since $R$ is kind, this implies $p(y^{-1}) = 0$ for all $y \in G_u\setminus\{u\}$. But $u$ was arbitrary, so $p(z) = 0$ for all $z \in G \setminus G^{(0)}$. Therefore $p \in D_R(G)$.
\end{proof}

\begin{corollary} \label{corr 1}
	Let $R$ be a kind, unital $*$-subring of $\CC$, and let $G$ and $H$ be ample groupoids. If $\Phi:~A_R(G) \to A_R(H)$ is a homomorphism of $*$-rings, then $\Phi(D_R(G)) \subseteq D_R(H)$.
\end{corollary}

\begin{proof}
	If $U \in \Bisc(G^{(0)})$ then $\bm{1}_U$ is a projection in $A_R(G)$. Since $\Phi$ is $*$-preserving, $\Phi(\bm{1}_U)$ is a projection in $A_R(H)$ and therefore it belongs to the diagonal subalgebra $D_R(H)$.
\end{proof}

\begin{corollary} \label{corr 2}
	Let $G$ and $H$ be Hausdorff ample groupoids such that $G$ is effective. The following are equivalent:
	\begin{enumerate}[\rm (1)]
		\item There is a topological isomorphism $\varphi: G \to H$,
		\item There is a $*$-ring isomorphism $\Phi: A_\ZZ(G) \to A_\ZZ(H)$,
		\item There is a ring isomorphism $\Phi: A_\ZZ(G) \to A_\ZZ(H)$ such that $\Phi(D_\ZZ(G)) = D_\ZZ(H)$,
		\item There is a ring isomorphism $\Phi: A_\ZZ(G) \to A_\ZZ(H)$ such that $\Phi(D_\ZZ(G)) \subseteq D_\ZZ(H)$.
	\end{enumerate}

	The statement remains true if $\ZZ$ is replaced by any kind, unital $*$-subring of $\CC$.\end{corollary}

\begin{proof}
	The equivalence of (1), (3), and (4) is proved in \cite[Corollary~5.8]{steinberg-new}. Clearly (1) implies (2). By Corollary \ref{corr 1},  (2) implies~(4).
\end{proof}

\section{Tensor products of Leavitt algebras} \label{s: leavitts}

We recall the definitions of Leavitt algebras \cite{leavitt1962module} and Cuntz groupoids~\cite{renault1980groupoid}. In short, $L_{n,R}$ is the Leavitt path algebra of the graph with one vertex and $n$ edges, and $\mathbb{O}_n$ is the boundary path groupoid of the same graph. It is well-known (see \cite{clark3,rigby}) that $A_R(\mathbb{O}_n) \cong L_{n,R}$, and the isomorphism restricts to $D_R(\mathbb{O}_n)  \cong D_{n,R}$.  Complete, classical definitions are given below.

\begin{definition}[Leavitt algebras]
	Let $R$ be a commutative unital ring, and let $2 \le n < \infty$. Define the \textit{Leavitt algebra of type $(1,n)$} to be the universal unital $R$-algebra $L_{n,R}$  generated by symbols $e_1, \dots, e_n, e_1^*, \dots, e_n^*$, subject to the relations:
	\begin{align*}
		e_i^* e_j = \delta_{ij}, &&
		\sum_{i = 1}^n{e_i e_i^*} = 1.
	\end{align*}
\end{definition}
	The \textit{diagonal subalgebra} of $L_{n,R}$ is the $R$-subalgebra $D_{n,R}$ generated by commuting idempotents of the form: \begin{align*}
	e_{i_1}\dots e_{i_k}e^*_{i_k} \dots e^*_{i_1}, && 1 \le i_1, \dots, i_k \le n.
	\end{align*}
		The algebra $L_{n,R}$ is universal among $R$-algebras $A$ with the property that $A^n \cong A$ as left $A$-modules. If $K$ is any field, $L_{n,K}$ is central, simple,  and infinite-dimensional. The diagonal subalgebra $D_{n,R}$ is a maximal commutative subalgebra of $L_{n,R}$ \cite[Corollary~2.4]{steinberg-new}. 
Given an involution $\overline{\phantom{f}}: R \to R$, there is a {canonical anti-linear involution}  $^*:L_{n,R} \to L_{n,R}$ given by the additive extension of the rule:
	\[
	re_{i_1} \dots e_{i_k}e^*_{j_{l}} \dots e^*_{j_{1}} \mapsto \overline{r}e_{j_1}\dots e_{j_l} e^*_{i_k} \dots e^*_{i_1}.
	\]

\begin{definition}[Cuntz groupoids]
	Let $2 \le n < \infty$. Let $W = \{1,\dots, n\}^\mathbb{N}$, and define the \textit{Cuntz groupoid}
	\[
	\mathbb{O}_n = \{(x,k,y) \in W \times \mathbb{Z}\times W\mid  x_{i+k} = y_i \text{ for all but finitely many } i \}.
	\]
	The unit space is $\{(x,0,x) \mid x \in W\}$, which is naturally identified with $W$, and the structure maps are:
	\begin{align*} 
		\src(x,k,y) = y, &&(x,k,y)^{-1} = (y,-k,x), \\
		\ran(x,k,y) = x,  &&
		(x,k,y)(y,l,z) = (x,k+l, z).\notag
	\end{align*}
	The topology on $\mathbb{O}_n$ is generated by the basis of open sets:
	\begin{align*}
		Z(a_1 \dots a_\ell, b_1\dots b_m) = \{(x,\ell-m,y) \in \mathbb{O}_n\mid \ x_1\dots x_\ell = a_1 \dots a_\ell,\ y_1 \dots y_m = b_1 \dots b_m \}
	\end{align*}
	where $a_1, \dots, a_\ell, b_1, \dots, b_m \in \{1, \dots, n\}$. 	With respect to this topology, $\mathbb{O}_n$ is a second-countable, $\sigma$-compact, minimal, and effective Hausdorff ample groupoid.
%
\end{definition}

%

%

Recall that a unital commutative ring is  \textit{indecomposable} if 0 and 1 are the only idempotents. For an example, take any integral domain.

\begin{proposition} \label{no-diag-iso}
	Let $R$ be indecomposable, and let $(n_1, \dots, n_k)$ and $(m_1, \dots, m_l)$ be increasing tuples of integers $\ge 2$. If there exists an isomorphism
	\[
	\varphi: \bigotimes_{i = 1}^kL_{n_i, R} \to \bigotimes_{j = 1}^l L_{m_j, R}
	\]
	 such that
	$\varphi
	\left(\bigotimes_{i = 1}^kD_{n_i,R}\right) \subseteq  \bigotimes_{j = 1}^l D_{m_j,R}$, then $k = l$ and $(n_1, \dots, n_k) = (m_1, \dots, m_k)$.
\end{proposition}

\begin{proof}
	Suppose that such an isomorphism $\varphi$ exists. Then there is an isomorphism
	\[
	\theta: \bigotimes_{i = 1}^k A_R(\mathbb{O}_{n_i})  \to \bigotimes_{j = 1}^l A_R(\mathbb{O}_{m_j})
	\]
	such that
	$
	\theta\left(\bigotimes_{i = 1}^nD_R(\mathbb{O}_{n_i})\right) \subseteq \bigotimes_{j = 1}^l D_R(\mathbb{O}_{m_j}).
	$
	By Theorem \ref{tensor}, there exists an isomorphism 
	\[
	\gamma:  A_R\left(\prod_{i = 1}^k\mathbb{O}_{n_i} \right) \to A_R\left(\prod_{j = 1}^l\mathbb{O}_{m_j}\right)
	\]	such that $\gamma\left(D_R\Big( \prod_{i = 1}^k\mathbb{O}_{n_i}\Big)\right) \subseteq  D_R\Big(\prod_{j = 1}^l\mathbb{O}_{m_j}\Big)$.
	The groupoid $\mathbb{O}_{n}$ is effective for any $n$, and a product of effective groupoids is effective. Then, \cite[Corollary~5.8]{steinberg-new} implies that $\mathbb{O}_{n_1} \times \dots \times \mathbb{O}_{n_k}$ and $\mathbb{O}_{m_1} \times \dots \times \mathbb{O}_{m_l}$ are isomorphic as topological groupoids.	Since every unit in $\mathbb{O}_{n_1} \times \dots \times \mathbb{O}_{n_k}$ has isotropy group $\ZZ^k$ and every unit in $\mathbb{O}_{m_1} \times \dots \times \mathbb{O}_{m_l}$ has isotropy group $\ZZ^l$, deduce that $k = l$. The adjacency matrix of the directed graph with one vertex and $n$ edges is the $1 \times 1$ matrix $[n]$. The Bowen-Franks group associated to $[n]$ is defined (see \cite{JS,matui}) as $\operatorname{BF}([n])=\ZZ/(n-1)\ZZ$. Therefore, $\operatorname{BF}([n_k]) \cong \operatorname{BF}([m_k])$ if and only if $n_k = m_k$. By \cite[Theorem~5.12]{matui}, $\mathbb{O}_{n_1} \times \dots \times \mathbb{O}_{n_k} \cong \mathbb{O}_{m_1} \times \dots \times \mathbb{O}_{m_l}$ implies $(n_1, \dots, n_k) = (m_1, \dots, m_k)$.
\end{proof}

From Corollary \ref{corr 2} and Proposition \ref{no-diag-iso}, we find that there are no unexpected $*$-isomorphisms between tensor products of integral Leavitt algebras.

\begin{corollary} \label{corr 3}
	Let $(n_1, \dots, n_k)$ and $(m_1, \dots, m_l)$ be increasing tuples of integers $\ge 2$. If there exists a $*$-ring isomorphism
	\[
	\varphi: \bigotimes_{i = 1}^kL_{n_i, \ZZ} \to \bigotimes_{j = 1}^l L_{m_j, \ZZ}
	\]
	then $k = l$ and $(n_1, \dots, n_k) = (m_1, \dots, m_k)$. The same statement holds if $\ZZ$ is replaced by any kind, unital $*$-subring of $\CC$.
\end{corollary}

In particular, there is no $*$-isomorphism $L_{2, \ZZ}\otimes L_{3, \ZZ} \to L_{2, \ZZ} \otimes L_{2, \ZZ}$, which is some partial progress on the question posed in \cite[Question~7.3.4.]{LPAbook}.


\end{document}